\documentclass[11pt]{article}

\usepackage{comment,amsmath,amssymb,amsthm,graphicx,fixmath}

\usepackage{algorithmic}
\usepackage{algorithm}
\usepackage{setspace}
\usepackage{ifthen}

\providecommand{\remove}[1]{}

\theoremstyle{plain}
\newtheorem{theorem}{Theorem}[section]
\newtheorem{lemma}[theorem]{Lemma}
\newtheorem{proposition}[theorem]{Proposition}

\newtheorem{mremark}[theorem]{Remark}

\newtheorem{observation}{Observation}

\theoremstyle{definition}
\newtheorem{definition}[theorem]{Definition}
\theoremstyle{remark}

\newcommand{\cardin}[1]{\lvert {#1} \rvert}

\newcommand{\us}{\chi_{\textrm{us}}}
\newcommand{\vr}{\chi_{\textrm{vr}}}
\newcommand{\lvr}{\chi_{\textrm{l-vr}}}

\newcommand{\pw}{\textrm{pw}}
\newcommand{\tw}{\textrm{tw}}

\newcommand{\alert}[1]{\textbf{\color{red}
[[[#1]]]}\marginpar{\textbf{\color{red}**}}\typeout{ALERT:
\the\inputlineno: #1}}

\usepackage{pdfsync}
\usepackage[colorlinks,linkcolor=blue,filecolor=blue,citecolor=blue,urlcolor
=blue,pdfstartview=FitH]{hyperref}

\newcommand{\namedref}[2]{\hyperref[#2]{#1~\ref*{#2}}}
\newcommand{\sectionref}[1]{\namedref{Section}{#1}}

\newcommand{\theoremref}[1]{\namedref{Theorem}{#1}}
\newcommand{\defref}[1]{\namedref{Definition}{#1}}
\newcommand{\figureref}[1]{\namedref{Figure}{#1}}

\newcommand{\remarkref}[1]{\namedref{Remark}{#1}}
\newcommand{\lemmaref}[1]{\namedref{Lemma}{#1}}

\newcommand{\propref}[1]{\namedref{Proposition}{#1}}
\newcommand{\obref}[1]{\namedref{Observation}{#1}}

\newcommand{\N}{\mathbb{N}}

\begin{document}

\title{On Vertex Rankings of Graphs and its Relatives}

\author{%
Ilan Karpas\thanks{
Mathematics department,Ben-Gurion University, Be'er Sheva 84105, Israel. 
\texttt{karpasi@math.bgu.ac.il}. Work on this paper was conducted while the author was an M.Sc student under the supervision of Shakhar Smorodinsky.
}
\and
Ofer Neiman\thanks{Computer Science department, Ben-Gurion University,
Be'er Sheva 84105, Israel.
\texttt{neimano@cs.bgu.ac.il}. Supported in part by ISF grant No. (523/12) and by the European Union�s Seventh Framework
Programme (FP7/2007-2013) under grant agreement $n^\circ 303809$.
}
\and
Shakhar Smorodinsky\thanks{Mathematics department, Ben-Gurion University,
Be'er Sheva 84105, Israel.
\texttt{shakhar@math.bgu.ac.il}, work on this paper by Shakhar Smorodinsky was supported by Grant 1136/12 from the Israel Science Foundation.
}
}

\date{}
\maketitle


\begin{abstract}
A vertex ranking of a graph is an assignment of ranks (or colors) to the vertices of the graph, in such a way that any simple path connecting two vertices of equal rank, must contain a vertex of a higher rank. In this paper we study a relaxation of this notion, in which the requirement above should only hold for paths of some bounded length $l$ for some fixed $l$. For instance, already the case $l=2$ exhibit quite a different behavior than proper coloring. We prove upper and lower bounds on the minimum number of ranks required for several graph families, such as trees, planar graphs, graphs excluding a fixed minor and degenerate graphs.
\end{abstract}


\section{Introduction}

A {\em vertex-ranking} of a graph $G=(V,E)$ with $k$ colors, is a $k$-coloring $c:V\to[k]$ such that for any $u,v\in V$, if $c(u)=c(v)$ then any simple path between $u,v$ contains a vertex $w$ with $c(w)>c(u)$ (note that this implies that the coloring is proper). The minimum $k$ such that $G$ admits a vertex-ranking with $k$ colors is called the \emph{vertex ranking chromatic number} (abbreviated vr-number), and is denoted by $\vr(G)$. In this paper we introduce a natural relaxation of vertex-ranking to the case where the above requirement holds only for paths of length bounded by $l$. In particular, we focus on the special case $l=2$.
\begin{definition}
Let $G=(V,E)$ be a simple graph. A \emph{unique-superior} coloring (abbreviated \emph{us-coloring}) of $G$ with $k$ colors is a function $c:V \to [k]$, such that $c$ is a proper coloring, and furthermore,
for every length two simple path $(u,v,w)$, if $c(u)=c(w)$ then $c(v)>c(u)$.
The minimum $k$ for which a graph $G$ admits a us-coloring with $k$ colors is called the \emph{unique-superior chromatic number} of $G$ (abbreviated \emph{us-number}), and is denoted by $\us(G)$.

More generally, in an  {\em $l$-vertex ranking} of $G$ with $k$ colors, every simple path of length at most $l$ between two vertices of the same color contains a vertex of higher color. The minimum $k$ for which $G$ admits an $l$-vertex ranking is called the {\em $l$-vr-number} of $G$, and is denoted by $\lvr(G)$.
\end{definition}

Note that trivially we have:
$$
\chi(G) \leq \us(G) \leq \vr(G).
$$

The vr-number is a well studied notion, with both theoretical and practical applications.
It has been used for VLSI lay-out design, Cholesky matrix factorization and scheduling
assembly problems (see, e.g., \cite{IRV88,L80,L90}). Recently, Even and Smorodinsky \cite{ES11} showed a strong
connection between vertex ranking of graphs and some online hitting set problems on graphs. Efficient polynomial algorithms to determine the vr-number of graphs have been found for several families of graphs: Sch\"{a}ffer showed this for trees \cite{S89} and Deugun et al. for permutation graphs \cite{DKKM94}. For other families of graphs, such as chordal graphs, bipartite and co-bipartite graphs, to name a few, the problem was shown to be NP-hard \cite{DN06,BDJ98}.

As mentioned above, for every graph $G$  $\chi(G)\leq \us(G) \leq \vr(G)$. However, these parameters can be arbitrarily far apart already for bipartite graphs. Indeed, consider for example the family of trees. It is a well known fact that every tree $T$ is bipartite i.e.,  $\chi(T)=2$. However, it is not a difficult exercise to show that for every path $P$ of length $n$ we have $\vr(P) = \Omega(\log n)$. See e.g., \cite{KMS95,Smor}. More generally, consider the family of planar graphs (a superset of the family of trees). By the famous Four-Color Theorem \cite{AH77} we have $\chi(G)\leq 4$ for every planar graph $G$. Katchalski et al. \cite{KMS95} proved that for any planar graph $G$ with $n$ vertices, $\vr(G)=O(\sqrt{n})$. They also showed that this bound is best possible, i.e., that for any integer $n$ there is a planar graph $G$ with $n$ vertices such that $\vr(G)= \Omega(\sqrt{n})$. Their proof extends to any fixed minor-free family. That is, a family of graphs excluding a fixed subgraph $H$ as a minor (see below for details). Another family of graphs for which the chromatic number and the vr-number can be far apart, is the family of $2$-degenerate graphs (or more generally $k$-degenerate graphs).
A graph is called $d$-degenerate if every subgraph has a vertex of degree at most $d$.
 It is well known and an easy fact to prove that a $d$-degenerate graph is always $d+1$-colorable. However, for any integer $n$ there are $2$-degenerate graphs $G$ with $n$ vertices such that $\vr(G) = \Omega(n)$.

\subsection{Related Notions of Coloring}

The notion of vertex-ranking is in itself a special case of the more general notion of \textit{unique-maximum} coloring of hypergraphs.\footnote{A hypergraph $H=(V,E)$, has vertices $V$ and hyperedges $E \subseteq 2^V$.}
A \textit{unique-maximum} coloring (abbreviated {\em um-coloring}) of a hypergraph $H$ is a coloring of its vertices such that for every hyperedge $e\in E$, the highest color of a vertex of $e$ is unique in $e$.
Then, for a given graph $G=(V,E)$, a vertex ranking of $G$ is exactly a um-coloring of the hypergraph $H=(V,E')$, where $E'$ is the family of all subsets of vertices that form a simple path in $G$.

A slightly different variant of the notion of um-coloring is that of \textit{conflict-free coloring}. This notion is a relaxation of um-coloring. In a conflict-free coloring of a hypergraph, the restriction is that for every hyperedge there is some color that occurs exactly once in that hyperedge. This coloring has been studied in many settings for the last decade, see the survey \cite{Smor} and the references therein.

There are other types of coloring that have been studied in the literature, which are more relaxed than us-coloring, yet stricter than proper coloring. Two such colorings are \textit{acyclic coloring} and \textit{star coloring}. An acyclic coloring of a graph $G$ is a proper coloring, such that, for every cycle of $G$, at least three colors appear in the cycle. Equivalently, its a proper coloring such that every subgraph induced by the union of two color classes is acyclic. A star coloring of a graph $G$ is a proper coloring, for which every path of length $4$ in $G$ uses at least $3$ colors. Equivalently, it is a proper coloring such that every subgraph induced by the union of two color classes is a collection of stars.
the star chromatic number and the acyclic chromatic number of $G$, $\chi_s(G)$ and $\chi_a(G)$, respectively, are defined in the obvious manner.
It is easy to see that we have the following hierarchy: Let $G$ be a graph, than $\chi(G)\leq \chi_a(G) \leq \chi_s(g) \leq \us(g) \leq \vr(G)$.
Gr\"{u}nbaum \cite{G73} showed that, for every planar graph $G$, $\chi_a(G)\leq 9$, and conjectured that $\chi(G) \leq 5$ for every planar graph, a conjecture which Borodin \cite{B79} later proved (this bound is known to be tight). A few years later, Ne\v{s}et\v{r}il and Ossona de Mendez \cite{NO03} showed that the acyclic chromatic number is bounded for every minor-excluding family.
Albertson et al. \cite{ACKKR04} proved that, for every graph $G$: $\chi_s(G)\leq 2\chi_a(G)^2+\chi_a(G)$, which means that the star coloring chromatic number is also bounded for graphs excluding a fixed minor.

\subsection{Our Results}
We provide lower and upper bounds on the us-number and the $l$-vr-number of several graph families. For trees we show a tight bound of $\Theta(\log n/\log\log n)$ on the us-number, where $n$ is the number of vertices. Thus already for trees the us-number can be far apart both from the chromatic number and the vr-number. For planar graphs, and more generally, graphs excluding a fixed minor, we show a nearly tight upper bound of $O(\log n)$ on the us-number and $O(l\cdot\log n)$ on the $l$-vr-number. For the family of $d$-degenerate graphs we show a lower bound of $\Omega(n^{1/3})$, and an upper bound of $O(\sqrt{n})$ on the us-number.

\section{Preliminaries}

All graphs mentioned in this paper are simple and undirected.
For a graph $G=(V,E)$ we
denote by $\Delta(G)$ (respectively, $\delta(G)$) its maximum (resp. minimum) degree. For a subset $V' \subseteq V$, $G[V']$ is the induced subgraph on $V'$, and $G\setminus V'$ is the graph induced on $V\setminus V'$. For a subgraph $H$ of $G$, we may write $G\setminus H$ which means $G\setminus V(H)$. For a graph $G=(V,E)$, a {\em subdivision} of an edge $\{u,v\}\in E$ means adding a new vertex $w$ on the edge which is connected only to $u$ and to $v$. In other words, it is the graph $G'=(V\cup\{w\},E\cup\{u,w\}\cup\{w,v\}\setminus\{u,v\})$.

\begin{definition} \label{d-deg}
A graph $G=(V,E)$ is called \textit{$d$-degenerate} if for every $V'\subseteq V$, $\delta(G[V']) \leq d$.
\end{definition}
Let $\mathbb{G}_d$ denote the family of \textit{$d$-degenerate graphs}.
It is easy to see that $\chi(G) \leq d+1$ for every $G \in \mathbb{G}_d$. The following propositions are well known, we include proofs for completeness.

\begin{proposition} \label{order}
Let $G=(V,E) \in \mathbb{G}_d$ be a graph with $n$ vertices. Then there is an ordering of the vertices  $(v_1,\dots,v_n)$,
so that for every $1 \leq i \leq n$, the vertex $v_i$ has at most $d$ neighbors $v_j$ with $j<i$.
\end{proposition}

\begin{proof}
By induction: let $v \in V$ be a vertex of degree at most $d$. By the induction hypothesis, the vertices of $G\setminus {v}$ can  be ordered $(v_1,\dots,v_{n-1})$ so that for all $1\le i\le n-1$, $v_i$ has at most $d$ neighbors to its left. Put $v_n=v$ and obviously the order $(v_1,\dots,v_n)$ satisfies the desired property.
\end{proof}

\begin{proposition}\label{ez-lemma}
Let $G=(V,E)$ be  a graph, $|V|=n$, $|E|=m$, and let $U \subseteq V$ be a a subset of vertices of degree at least $d$.
Then $|U| \leq 2m/d$.
\end{proposition}
\begin{proof}
By summing the degrees of the vertices of $G$, we have:
\[
2m=\Sigma_{v \in V}d(v) \geq \Sigma_{v \in U}d(v) \geq d|U|~,
\]
so that $|U| \leq 2m/d$.
\end{proof}

The {\em square} of a graph $G=(V,E)$, denoted by $G^2=(V,E')$, is the graph on the vertex set $V$ with $\{u,v\}\in E'$ if and only if $d_G(u,v)\in\{1,2\}$ where $d_G(u,v)$ denotes the length of a shortest path between $u$ and $v$. Observe that a proper coloring of $G^2$ is a us-coloring of $G$ (because there are no paths of length two whose end-points may get the same color).

\section{The us-number of Trees}

In this section we focus on the us-number of trees, and provide sharp asymptotic bounds.
\begin{theorem}
For every tree $T$ on $n$ vertices,
\[
\us(T)=O(\frac{\log{n}}{\log{\log{n}}})~.
\]
\end{theorem}

\begin{theorem}
For any integer $n$ there exists a tree $T$ with $n$ vertices such that $\us(T)=\Omega(\frac{\log{n}}{\log{\log{n}}})$
\end{theorem}
\subsection{Upper bound}

Let $k > 3$ be some integer. It suffices to show that if a tree $T$ has $\us(T)=k$ then $|V(T)|\ge \Omega((k-2)!)$, because using standard bounds on $k!$ we have that $\log n = \Omega(k\log k)$. To this end, define for each $1\le i\le k-1$ a boolean function $p_i$ on the set of all rooted trees as follows: For a tree $T$ with root $r$, $p_i(T,r)=1$ if and only if for every two colors $i\leq l, m \leq k-1$, there is a us-coloring of $T$ with $k-1$ colors, so that $r$ has color $l$, and none of its children have color $m$. Because to have $p_i(T,r)=1$ it must be that $T$ admits a us-coloring with $k-1$ colors, we have the following:
\begin{observation}\label{ob:pi}
If $\us(T)=k$ then $p_i(T,r)=0$ for every $i$ and every $r\in V(T)$.
\end{observation}
Define the function $f:[k-1]\to\N$ by
\[
f(i)=\min\{n:\textrm{ there exists a tree $T$ on $n$ vertices with root $r\in V(T)$ such that } p_i(T,r)=0\}~.
\]
By \obref{ob:pi}, if $\chi_{us}(T)=k$ then $|V(T)| \geq f(k-1)$, thus proving the following Lemma will conclude the proof.
\begin{lemma}
$f(k-1)>(k-2)!$ .
\end{lemma}

\begin{proof}
We show that for every $1 \leq i \leq k-3$,
\[
f(i) \geq 1+(k-i-1)f(i-1)~.
\]
Let $T$ be a tree with root $u$, such that $p_i(T,u)=0$ and $|V(T)|=f(i)$.
Denote by $u_1, \dots u_s$ the children of $u$, and by $T_{u_j}$ the sub-tree of $T$ rooted in $u_j$.
Seeking contradiction, assume that $f(i)< 1+(k-i-1)f(i-1)$. Then there are at most $k-i-2$ trees $T_{u_j}$ for which $p_{i-1}(T_{u_j}, u_j)=0$, since
the number of vertices in each such tree is at least $f(i-1)$. By renaming we may assume that for $1 \leq j \leq s'$, $p_{i-1}(T_{u_j},u_j)=0$, and for $s'+1 \leq j \leq s$, $p_{i-1}(T_{u_j},u_j)=1$, where $s' \leq k-i-2$.

Fix any $i \leq l,m \leq k-1$. We demonstrate that there is a us-coloring $c$ of $T$ with at most $k-1$ colors, so that $c(u)=l$ and $c(u_j) \neq m$ for all $j\in[s]$. First, for each $j>s'$, as $p_{i-1}(T_{u_j},u_j)=1$ there is a coloring $c_j$ of $T_{u_j}$ such that $c_j(u_j)=i-1$ and $c_j(w) \neq l$ for every child $w$ of $u_j$. As for $1 \leq j \leq s'$, we use the fact that $T$ has the minimal number of vertices for which $p_i(T,u)=0$, thus it must be that $p_i(T_{u_j},u_j)=1$. This implies that for every color $d$ out of the color set $A=\{i, \dots k-1\} \setminus \{l,m\}$ there exists a us-coloring $c_j$ of $T_{u_j}$ so that $c_j(u_j)=d$ and $c_j(w) \neq l$ for every child $w$ of $u_j$. As $|A|=k-i-2$, we can choose a color $d(j)$ such that $d(j)\neq d(j')$ for all $s'<j<j'\le s$. Finally, color the root $u$ with color $l$. See Figure~\ref{some-label} for an illustration. Then it can be verified that the obtained coloring $c$ is a valid us-coloring of $T$ satisfying the constraints imposed by the definition of $p_i(T,r)$. But this means that $p_i(T,u)=1$, a contradiction.
Using that recursive formula, we have $f(k-3)>(k-3)!f(1)$. One can check that $f(1)=k-2$, so $f(k-3)>(k-2)!$, and of course also $f(k-1)>(k-2)!$ .
\end{proof}

\begin{figure}
    \centering
    \includegraphics[width=5in]{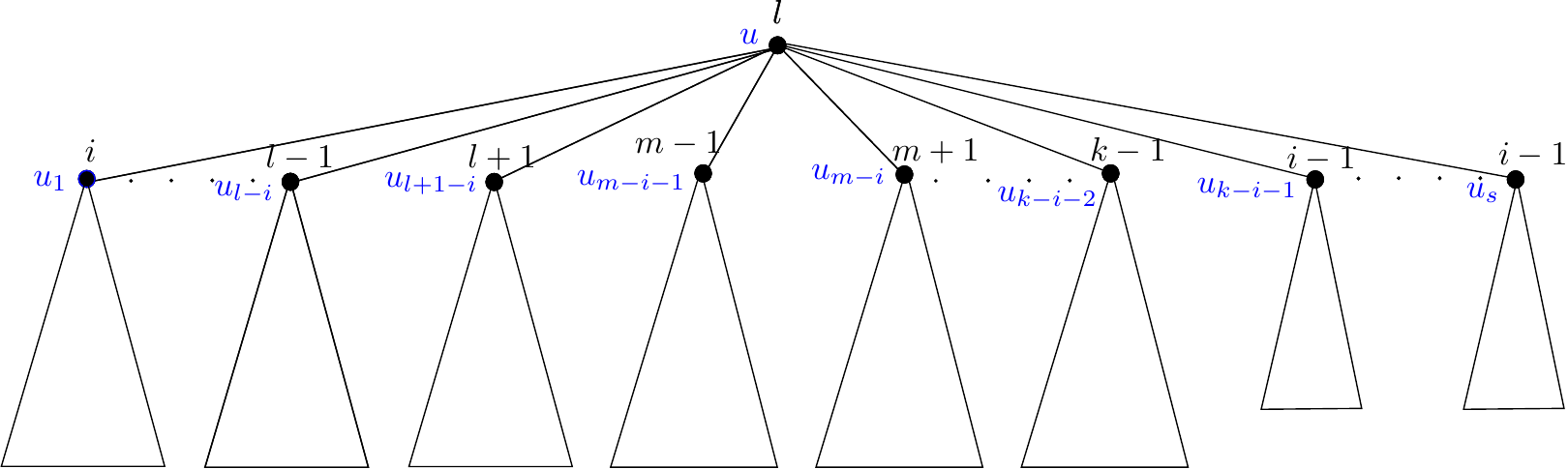}
    \caption{\label{some-label} $s'=k-i-2$, and $l\le m$. Not shown in the figure are the grandchildren of $u$, all of which have colors different from l.}
\end{figure}

\subsection{Lower bound}
Next, we show that for every $n$, there exist a tree $T$
on $n$ vertices such that $\us(T) = \Omega(\frac{\log n}{\log \log n})$.
Let $T=T_k$ be the \textit{complete $k$-ary tree with $k$ levels}, i.e., a rooted tree of height $k-1$, where each non-leaf has exactly $k$ \textit{children}, and all leaves have distance $k-1$ from the root. Since $|V(T_k)|=\frac{k^k-1}{k-1}$ it follows that $k=\Omega(\frac{\log n}{\log \log n})$. The {\em level} of a vertex is its distance to the root, so the root is at level $0$, its children
are at level $1$, and the leaves are at level $k-1$. See Figure \ref{Tree} for an illustration with $k=3$. We shall prove that $\us(T_k)=k=\Omega(\frac{\log n}{\log \log n})$.

\begin{figure}[h]
    \centering
    \includegraphics[width=3in]{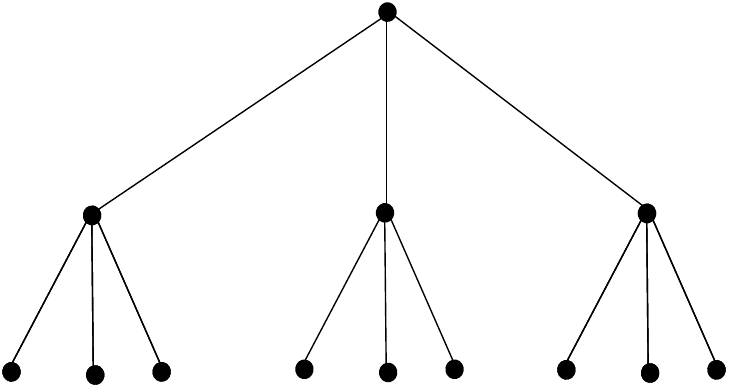}
    \caption{\label{Tree} $T_3$.}
\end{figure}

\begin{lemma}\label{tree-lemma}
Let $c$ be a us-coloring of $T_k$ with $k$ colors. Then for every
$1 \leq i \leq k$, color $i$ can be used only for vertices from the bottom $i$ levels.
\end{lemma}

\begin{proof}

The proof is by induction on $i$. The base case is for $i=1$, and seeking contradiction, assume that $c(v)=1$ and $v$ is not a leaf. Note that $v$ has exactly $k$ children $u_1,\dots,u_k$. Since $c(v)=1$ and the coloring is proper, then $c(u_l)>1$, for all
$l \in [k]$. Furthermore, $c(u_l) \neq c(u_m)$ for every $1 \leq l < m \leq k$, since $v$ is a common neighbor of $u_l$ and $u_m$ with the lowest color. Hence, at least $k+1$ colors are used, a contradiction.

The induction hypothesis is that the statement holds for all $j<i$. Seeking contradiction, assume that $c(v)=i$ for some vertex $v \in V(T_k)$ which is not in the bottom $i$ levels. Denote by $u_1,\dots,u_k$ the children of $v$. We know that $c(u_l)>i-1$ for every $1 \leq l \leq k$ by the induction hypothesis, since $u_l$ is not in the bottom $i-1$ levels. Moreover, for every $l \in [k]$, since $c$ is also a proper coloring then $c(u_l) \neq c(v)=i$,
and we conclude that $c(u_l)>i$.
This means that $c(u_l) \neq c(v)$ and $c(u_l) \neq c(u_m)$ for every $1 \leq l < m \leq k$. Again, this implies that $k+1$ colors are used for $v,u_1,\dots,u_k$, a contradiction. This completes the induction step, proving the lemma.
\end{proof}

Now we are ready to prove that $\us(T_k)=k$.
It can be checked that $\us(T_k) \leq k$, since the coloring which assigns every vertex of level $i$ the color $k-i$ is a us-coloring of $T_k$.
To see that $\us(T_k) \geq k$, we note that applying \lemmaref{tree-lemma} on the root $r$ suggests that it can be assigned only colors which are at least $k$.

\section{Planar Graphs}\label{sec:planar}

In this section we provide a nearly sharp asymptotic upper-bound on the us-number and the $l$-vr-number of planar graphs.

\subsection{The us-number of Planar Graphs}

\begin{theorem} \label{thm:planar}
For any planar graph $G$ on $n$ vertices, $\us(G) = O(\log n)$.
\end{theorem}

Before presenting the proof of \theoremref{thm:planar}, we give an informal sketch.
Given a planar graph $G=(V,E)$, we begin by partitioning the set of vertices into $t=O(\log n)$ independent sets
$V = V_1\cup \cdots \cup V_t$ with the following property: For every $1\leq i < j \leq t$, any vertex $v \in V_i$ has at most $O(1)$ neighbors in $V_j$. For each of the sets $V_j$ we shall use a distinct palette of $O(1)$ colors, such that $V_i$ will have smaller colors than $V_j$ for all $1\le i<j\le t$. We shall exhibit a coloring of each set $V_j$ with the property that whenever two vertices $u,v \in V_j$ have a common neighbor in $V_i$ (for $i < j$) then $u$ and $v$ will get distinct colors.

We need the following crucial lemma which provides a coloring with the required property for each of the $V_j$.

\begin{lemma} \label{bipartite}
Let $G=(A \cup B, E)$ be a bipartite planar graph with bipartition $A$ and $B$. Assume further that every vertex in $A$ has degree at most $b$ for some fixed integer $b$. Then there exists a coloring $c:B \rightarrow [6b]$ of the vertices in $B$,
such that if $u,v\in B$ have a common neighbor (in $A$) then $c(u)\neq c(v)$.
\end{lemma}

\begin{proof}
By a lemma that can be found in, e.g., \cite{D02}, there exists a plane graph $G'=(B,E')$ with a set of faces $F$ and a bijection $g:A\to F'$ where $F' \subset F$, that satisfies the following property: For any two vertices $x \in A$, $y \in B$ we have that $\{x,y\}\in E(G)$ if and only if $y$ is incident in $G'$ with the face $f=g(x)$.
Our goal is to show that there exists a coloring of the vertices of $G'$ with $6b$ colors, such that all vertices incident on a face $f\in F'$ get distinct colors. We define the auxiliary graph $G''$ by adding edges to $G'$. For every $f\in F'$ we place edges between any pair of vertices on $f$.

It remains to prove that $\chi(G'') \leq 6b$. We do this by showing that $G''$ is $(6b-1)$-degenerate.
Note that $G''$ is obtained from $G'$ by adding cliques to a subset of the faces. Since the degree in $G$ of any vertex in $A$ is at most $b$, each of the faces in $F'$ has at most $b$ vertices on its boundary. As each clique of size $b$ is contained in a union of $b-1$ trees\footnote{If the vertices of the clique are $v_1,\dots,v_b$, then the tree $T_i$ for $1\le i\le b-1$, will consist of all edges from $v_i$ to the other vertices.}, we have that $G''$ is the union of at most $b-1$ planar graphs. The same argument holds for any subgraph of $G''$, thus, as the average degree of a planar graph is bounded by $6$, the average degree in a union of $b-1$ planar graphs is at most $6(b-1)$, so $G''$ is indeed $(6b-1)$-degenerate and so $6b$ colorable.
\end{proof}

\begin{proof}[Proof of \theoremref{thm:planar}]
We start by partitioning the vertices into $t=O(\log n)$ independent sets as follows:
Let $V' \subset V$ be the subset of vertices with degree less than $12$. It is easy to see that $\cardin{V'} \geq n/2$. Indeed, by Euler's formula it follows that $|E(G)| \leq 3n-6$, and by \propref{ez-lemma}, $\cardin{V \setminus V'} < 6n/12$ and hence $\cardin{V'} \geq n/2$.
Since every planar graph is four-colorable \cite{AH77}, there exists a subset $V'' \subset V'$ of size at least $\frac{\cardin{V'}}{4} \geq n/8$ which is also independent. Put $V_1= V''$ and continue recursively on the subgraph $G[V\setminus V_1]$.
Since in every iteration we discard at least an $1/8$ fraction of the remaining vertices, the final number of sets is at most $t=\log_{8/7}n+1$.

Next, we color each of the sets $V_i$ by applying \lemmaref{bipartite} on the bipartite graph induced by $A=\bigcup_{j=1}^{i-1}V_j$ and $B=V_i$ (that is, discarding the internal edges in $A$ and $B$). Observe that by construction every vertex $y\in A$ has at most $11$ neighbors in $B$, thus there exists a coloring of $V_i$ with $66$ colors, such that if $x,z\in B$ have a common neighbor $y\in A$, then $x,z$ get different colors. The total number of colors used for $V$ is at most $66t= O(\log n)$.

To complete the proof of the theorem, we need to show that the coloring $c$ obtained in this manner is a us-coloring. Indeed, the coloring is proper since each $V_i$ is an independent set, and each $V_i$ has different colors than $V_j$ for $i\neq j$.
Thus, it is only left to check that if we have a path $(x,y,z)$ for which $c(x)=c(z)$ then $c(y) > c(x)$. Since $c(x)=c(z)$ it must be that $x,z\in V_i$ for some $1\le i\le t$, and by \lemmaref{bipartite} it also holds that $y\in V_j$ for some $j\ge i$ (as otherwise the assertion of the Lemma is violated for this triple). Recall that $V_i$ is an independent set, so in fact $j>i$, and then $c(y)>c(x)$ by construction.
\end{proof}

\subsection{The $l$-vr-number of Planar Graphs}
Here we generalize \theoremref{thm:planar} to the $l$-vr-number of planar graphs. The main tool is a planar separator theorem due to \cite{T04}, where the separator consists of three shortest paths.

\begin{definition}
Let $G=(V,E)$ be a graph. An $\alpha$-separator of $G$ is a set of vertices $U \subseteq V$, such that the number of vertices in any connected component of $G\setminus U$ is at most $\alpha|V|$.
\end{definition}

Let $T$ be a rooted tree with root $r$, and let $u$ be some vertex in $T$. Denote by $P_T(u)$ the unique simple path in $T$ from $r$ to $u$.
Before introducing the main theorem of this section, we need the following separator theorem of Thorup~\cite{T04}, based on the planar separator of \cite{LT79}:
\begin{theorem}[\cite{T04}]
\label{paths-separator}
Let $G=(V,E)$ be a connected planar graph with $n$ vertices and let $T$ be some spanning tree of $G$ with root $r$.
Then there exist three vertices $u,v,w$ such that $P_T(u) \cup P_T(v) \cup P_T(w)$ is a $1/2$-separator of $G$.
\end{theorem}

The result of this section is the following theorem.

\begin{theorem}\label{thm:vr-l-planar}
Let $l$ be some fixed positive integer.
Then for any $n>0$ and any planar graph $G=(V,E)$ with $n$ vertices,
$\lvr(G)\leq \lceil 3(l+1)\log n\rceil$. Furthermore, computing such a coloring can be done in polynomial time.
\end{theorem}

\begin{proof}
The proof is by induction on $n$.
Fix an arbitrary vertex $r \in V$. Let $T$ be a Breadth-First-Search (BFS) tree of $G$ rooted at $r$.\footnote{A BFS tree with root $r$ is the tree created by taking the shortest paths from every node $u$ to $r$ (assuming a consistent choice between paths of equal length).}
By \theoremref{paths-separator} there exist three vertices $u_0,u_1,u_2\in V$ such that the set $S=P_T(u_0) \cup P_T(u_1) \cup P_T(u_2)$ is a $1/2$-separator of $G$.  Every connected component $Q$ in the graph $G\setminus S$
has at most $n/2$ vertices, so by the induction hypothesis $\lvr(Q)\leq\lceil 3(l+1)\log(n/2)\rceil=\lceil 3(l+1)\log n\rceil-3(l+1)$. Let $d=\lceil 3(l+1)\log n\rceil-3(l+1)$, and for each such connected component $Q$, use the colors $[d]$ to color it inductively (note that we use the same color set for different components). It remains to show that we can complete such a coloring to a legal $l$-vertex ranking of $G$ using additional $3(l+1)$ colors. For $j\in\{0,1,2\}$ and a vertex $x\in P_T(u_j)$ with $d_G(x,r)=m$, use color number
\[
d+1  + m~(\textrm{mod }l+1)+(l+1)j~.
\]
This implies that vertices on $P_T(u_0)$ will have the colors $d+1,\dots,d+l+1$, vertices on $P_T(u_1)$ will have the colors $d+l+2,\dots,d+2l+2$, and vertices on $P_T(u_2)$ will have the colors $d+2l+3,\dots,d+3l+3$.(If a vertex belong to more than one of $P_T(u_j)$, it can use any one of the colors given to it).

To see that this is indeed an $l$-vertex ranking of $G$, let $s, t \in V$ be two vertices with the same color, and let $P$ be a path between $s$ and $t$ of length at most $l$. Notice that either both $s,t$ are in $V\setminus S$ or both vertices are in $S$ (for otherwise, they cannot get the same color). If the former case is true, then there are two subcases: Either $P$ is entirely contained in $V \setminus S$, thus it lies inside a single connected component of $G\setminus S$, and by the induction hypothesis $P$ contains a vertex with color greater than the color of $s$ and $t$. The other subcase is that there exist some vertex $x \in P \cap S$, but then the color of $x$ is strictly greater than $d$, which is the bound on the colors of $s$ and $t$.
If, on the other hand, the latter case holds and $s,t \in S$, then it must be the case that both $s$ and $t$ belong to the same path $P_T(u_j)$ (otherwise they can not have the same color). By a property of a BFS tree, the path between $s$ and $t$ in $T$ is also the shortest path between $s$ and $t$ in $G$. Yet $s$ and $t$ have the same color, thus by construction the path between them in $T$ is of length at least $l+1$, which is a contradiction.

To see that such a coloring can be obtained in polynomial time, we rely on the fact that finding a BFS tree can be done in linear time. Additionally, by \cite{T04}, finding the three paths originating from the root can be done in polynomial time. Finally, the recursion occurs at most $\log n$ times, since in each step the number of vertices in each connected component is at most half that of the graph in the previous step.
\end{proof}

\section{Minor-Excluded Graph Families}

In this section we extend the results of \sectionref{sec:planar} to families of graphs excluding some fixed minor.

\subsection{The us-number of Graphs Excluding a Fixed Minor}

As our bounds depend only on the cardinality of the minor being excluded, it suffices to consider the family ${\cal M}_r$, consisting of all graphs that exclude $K_r$ as a minor. We begin with a bound on the us-number of such graphs, which generalizes \theoremref{thm:planar}.\footnote{Recall that planar graphs are in particular $K_5$-free.}

\begin{theorem} \label{mf-upper}
For any graph $G \in {\cal M}_r$ with $n$ vertices,
\[
\us(G)=O((r\sqrt{\log r})^3\cdot\log n)~.
\]
\end{theorem}
The structure of the proof is the same as that of \theoremref{thm:planar} (though the constants will now depend on $r$). We will need to formulate and prove a lemma similar to \lemmaref{bipartite}.

\begin{lemma} \label{bipartite-minor}
Let $G=(A \cup B, E)$ be a bipartite graph in ${\cal M}_r$ with bipartition $A$ and $B$. Assume further that every vertex in $A$ has degree at most $b$ for some fixed integer $b$. Then there exists a coloring $c:B \rightarrow [Cb]$ of the vertices in $B$, for some $C=O(r\sqrt{\log r})$,
such that if $u$ and $v$ are two vertices in $B$ with a common neighbor (in $A$) then $c(u)\neq c(v)$.
\end{lemma}

\begin{proof}
W.l.o.g assume that each vertex in $A$ has degree exactly $b$, and for each vertex $v \in A$, fix some ordering of its $b$ neighbors $\{u_1,\cdots,u_b\}$. For each $1\leq i \leq b$, let $E_i\subseteq E$ be the set of edges between each $v\in A$ and its $i$-th neighbor. Define a graph $G_i$, which is obtained from $G$ by contracting all edges in $E_i$. Observe that we contract exactly one edge for each $v\in A$, thus the vertex set of $G_i$ can be identified with $B$. Moreover, the graph $G_i$ excludes $K_r$ as a minor, since excluding minors is invariant under edge contraction.

Consider the graph $G'=(B,\cup_{i=1}^b E(G_i))$. We claim that a proper coloring $c$ of $G'$, will also be a coloring satisfying the assertion of the Lemma. Indeed, if two vertices $x,y\in B$ have a common neighbor $v \in A$, then if $x$ is the $i$-th neighbor of $v$, the edge $\{y,v\}$ is not contracted in $G_i$. Observe that the contraction of $\{v,x\}$ in $G_i$ will have $x,y$ as neighbors, so a proper coloring will provide them with different colors. What remains to be proven, then, is that $G'$ can be properly colored with $O(br\sqrt{\log n})$ colors. By a theorem of Thompson \cite{T01}, a $K_r$-minor-free graph has an average degree of $O(r\sqrt{\log r})$. Since $G'$ (and any subgraph of $G'$) is the union of $b$ such graphs, it has an average degree at most $O(br\sqrt{\log r})$, hence it is $O(br\sqrt{\log r})$-degenerate. As we mentioned earlier, this means that $G'$ can be properly colored with $O(br\sqrt{\log r})$ colors.
\end{proof}

\begin{proof}[Proof of Theorem \ref{mf-upper}]
We go along similar lines as in the case for us-coloring of planar graphs, create a partition of $V$ to independent sets $V_1\cup\cdots\cup V_t$, such that for every $1\le i\le t$, any $x\in V_i$ has at most $b=b(r)$ neighbors in $\cup_{j=i}^tV_j$.

By \cite{T01} the average degree of $G$ is less than $a\cdot r\sqrt{\log r}$ (where $a$ is a universal constant), so let $b=2a\cdot r\sqrt{\log r}$. Let $V'\subseteq V$ to be the set of vertices whose degree is at most $b$, and by \propref{ez-lemma} we get $|V'|\ge n/2$. As $G[V']$ is $(a\cdot r\sqrt{\log r}-1)$-degenerate, it can be properly colored using $a\cdot r\sqrt{\log r}$ colors, so there is a color class $V''$ of size at least $|V''|\ge |V'|/(a\cdot r\sqrt{\log r})\ge n/b$. Put $V_1=V''$ and continue recursively on the graph $G[V\setminus V_1]$. At each step we throw out at least $\frac{1}{b}$ fraction of the vertices, which means that $t=O(\log_{1/(1-1/b)} n)=O(b\log n)$.

Now, for each $1\le i\le t$, color $V_i$ with $O(b\cdot r\sqrt{\log r})$ colors, by using \lemmaref{bipartite-minor} on $A=\cup_{j=1}^{i-1}V_j$, $B=V_i$ and the parameter $b$. From the same considerations as in the proof of \theoremref{thm:planar}, we obtain a us-coloring of $G$ of size $O(tb\cdot r\sqrt{\log r})=O((r\sqrt{\log r})^3\cdot\log n)$.
\end{proof}

\subsection{The $l$-vr-number of Graphs Excluding a Fixed Minor}

In this section we extend all theorems in this section and \sectionref{sec:planar} (albeit, with worse constants), by proving a bound on the $l$-vr number of graphs excluding a fixed minor. We shall use {\em path separators} of such graphs due to \cite{AG06}, which extend the planar separators of \cite{LT79,T04}.

\begin{definition}\label{def:path-separator}
A graph $G=(V,E)$ on $n$ vertices is {\em $s$-path separable} if there exists an integer $t$ and a separator $S\subseteq V$ such that:
\begin{enumerate}
\item $S=P_0\cup P_1\cup\dots P_t$, where for each $0\le i\le t$, $P_i$ is a collection of shortest paths in the graph $G\setminus(\bigcup_{0\le j<i}P_j)$.

\item $\sum_{i=0}^t|P_i|\le s$, that is, the total lengths of paths in $S$ is at most $s$.

\item Each connected component of $G\setminus S$ is $s$-path separable and has at
most $n/2$ vertices;
\end{enumerate}
\end{definition}

\begin{theorem}[\cite{AG06}]\label{thm:path-separator}
Every $H$-minor-free graph is $s$-path separable, for $s = s(H)$, and
an $s$-path separator can be computed in polynomial time.
\end{theorem}

Now we are ready to prove the following theorem. The proof is similar to the proof of \theoremref{thm:vr-l-planar}, and we give the full details for completeness.

\begin{theorem}\label{thm:vr-l-minor}
Let $G=(V,E)$ be a graph on $n$ vertices that excludes $H$ as a minor, then $\lvr(G)\le s (l+1)\log n$, where $s=s(H)$ is a constant that depends only on $H$. Moreover, such a coloring can be computed in polynomial time.
\end{theorem}
\begin{proof}
Let $S$ be an $s$-path separator of $G$ for some $s=s(H)$, as guaranteed by \theoremref{thm:path-separator}. As each connected component of $G\setminus S$ has at most $n/2$ vertices (and of course still excludes $H$ as a minor), we assume inductively that there exists an $l$-bounded vertex ranking for it using $s(l+1)\log(n/2)=s(l+1)\log n-s(l+1)$ colors. Note that we use the same colors for different components. We will use additional $s(l+1)$ colors for the vertices of $S$, each of these new colors will be higher than each color used for $V\setminus S$. Each of the paths in $S$ will have its own $l+1$ colors, in such a way that paths in $P_j$ will have higher colors than paths in $P_i$ for all $0\le j<i\le t$. Note that the second property of \defref{def:path-separator} guarantees that we have enough colors. To color a path, simply apply the $l+1$ colors consecutively in a cyclic manner along the path as in \theoremref{thm:vr-l-planar}. If a certain vertex belongs to more than one path, it will keep the highest color assigned to it.

Next we prove the validity of the coloring. To this end, consider a path $P$ of length at most $l$. If it is the case that $P$ is contained in a connected component of $G\setminus S$, the induction hypothesis guarantees that the coloring is valid, so consider the case that $P\cap S\neq\emptyset$. Let $i$ is the minimal such that $P\cap P_i\neq \emptyset$. By the definition of the coloring, the highest color in $P$ will be one the colors of $P_i$. Let $Q$ be the path in $P_i$ such that the highest color in $P$ comes from a vertex of $Q$. The proof will be concluded once we show that any color in $P\cap Q$ is unique. Seeking contradiction, assume that there are two vertices $u,v\in P\cap Q$ that were assigned the same color. By the first property of \defref{def:path-separator}, $Q$ is a shortest path of $G'=G\setminus (\bigcup_{0\le j<i}P_j)$, so the distance between $u,v$ in $G'$ must be at least $l+1$. However, the minimality of $i$ suggests that $P$ is also contained in $G'$, and as $P$ is a path of length at most $l$, it cannot contain both $u,v$.
\end{proof}

\section{Degenerate Graphs}

In this section we focus on the \textit{us-number} of the family of $d$-degenerate graphs. This family is a generalization of ${\cal M}_r$, and it could seem plausible that the us-number will be at most polylogarithmic in the number of vertices for degenerate graphs as well. We show that no such bound exists, by showing that for any integer $n$ there exists a $2$-degenerate graph $G$ with $n$ vertices satisfying $\us(G)=\Omega(n^{1/3})$. We also prove an upper bound $\us(G)=O(d\sqrt{n})$ for any $G\in\mathbb{G}_d$ with $n$ vertices.
These bounds demonstrate the large difference between us-coloring and proper coloring. We also show a large gap between the us-number and vr-number, by proving that there exists a graph $G\in\mathbb{G}_3$ with $n$ vertices such that $\vr(G)=\Omega(n)$.

\subsection{Upper Bound on the us-number of Degenerate Graphs}

In order to bound the us-number we will need the following Lemma, that provides a bound on the (usual) chromatic number of the squared graph $G^2$, for degenerate bounded degree graphs.

\begin{lemma} \label{d-lemma}
Let $G \in \mathbb{G}_d$ be a graph with maximum degree $\Delta$. Then $\chi(G^2) \leq (2\Delta+1)d+1$.
\end{lemma}
\begin{proof}
Take an ordering of the vertices of $G$, $V= (v_1,\dots,v_n)$ as in \propref{order}. We color the vertices greedily in order. Assume by induction that we have a proper coloring of $G^2[\{v_1,\dots,v_{i-1}\}]$ using at most $(2\Delta+1)d+1$ colors. Assign to $v_i$ the lowest possible color $c_i$.
We claim that $c_i \le (2\Delta+1)d+1$. Indeed, since the maximum degree is $\Delta$ there can be at most $\Delta$ indices $h>i$ such that $v_h$ is a neighbor of $v_i$. By the property of \propref{order}, each $v_h$ can have at most $d$ neighbors $v_j$ with $j<h$, so there are at most $\Delta d$ paths of the form $(v_i,v_h,v_j)$ with $j<i<h$. On the other hand, $v_i$ itself has at most $d$ neighbors $v_j$ with $j<i$, each of these has at most $\Delta$ neighbors, so there are at most $\Delta d$ paths of the form $(v_i,v_j,v_g)$ with $j<i$, $g<i$. We conclude that there are at most $2\Delta d$ vertices among $v_1,\dots v_{i-1}$ that are of distance $2$ from $v_i$, and at most $d$ of distance $1$, therefor $v_i$ has an available color in $[(2\Delta+1)d+1]$ such that the coloring of $G^2[\{v_1,\dots,v_i\}]$ is proper.
\end{proof}

\begin{theorem}
 Let  $G \in \mathbb{G}_d$ be a graph with $n$ vertices. Then $\us(G)=O(d\sqrt{n})$.
 \end{theorem}

\begin{proof}
Put $U=\{v \in V: d(v)< \sqrt{n}\}$, and let $G[U]$  be the subgraph of $G$ induced by $U$. Notice that $G[U] \in \mathbb{G}_d $   and $\Delta(G[U])< \sqrt{n}$, so by \lemmaref{d-lemma}, $\us(G[U])\le \chi(G[U]^2)\le (2\sqrt{n}+1)d$. Assume that we have such a coloring for $U$, and it remains to complete the coloring for the large degree vertices $V\setminus U$. For every vertex $v \in V \setminus U$ use a distinct color larger than  $(2\sqrt{n} +1)d $. By \lemmaref{ez-lemma},
$|V \setminus U| \leq \frac{2dn}{\sqrt{n}}=2d\sqrt{n}$, so in total we have used at most $d(4\sqrt{n} +1)$ colors.
\end{proof}

\subsection{Lower Bound on the us-number of Degenerate Graphs}
\begin{theorem}
  For  every $n$ there is a \textit{$2$-degenerate graph} $G=(V,E)$ with $n$ vertices, such that $\us(G)>n^{1/3}$.
\end{theorem}

\begin{proof}
Fix some $k \in \mathbb{N}$, and let $G$ be the graph obtained by taking the complete graph $K_k$, replicating each edge $k$ times and then subdividing each edge.
Notice that $G$ is, indeed, a $2$-degenerate graph, since every vertex subdividing an edge has degree $2$, and once all these vertices are removed we are left  only with isolated vertices. We have  $|V|=k+  k {k\choose 2}<k^3$. We will prove that $\us(G) \geq k$. Seeking contradiction, assume that $c:V \to \{1,\dots,m\}$ is a unique-superior coloring of $G$, with $m<k$. Then, by the pigeon-hole principle, there must be two distinct vertices among the original $k$ vertices of $K_k$, $v_i$ and $v_j$, such that $c(v_i)=c(v_j)$. Hence, for every vertex $u$ subdividing an edge between $v_i$ and $v_j$, $c(u)>c(v_i)$. But notice that it must hold that $c(u) \neq c(u')$ for every two vertices $u$ and $u'$ subdividing two distinct edges between $v_i$ and $v_j$, as otherwise the path $(u,v_i,u')$ will violate the requirement of a legal us-coloring. Since there are $k$ vertices subdividing two distinct edges between $v_i$ and $v_j$, we must use at least $k$ distinct colors for them, which yields a contradiction. See \figureref{DED} for an illustration.
\end{proof}

\begin{figure} [H]  \label {DED}
    \centering
    \includegraphics[width=3in]{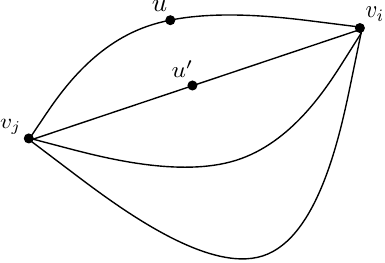}
    \caption{if $c(v_i)=c(v_j)$, then  $c(u) \neq c(u')$ for every two subdividing vertices $u$ and $u'$.  }
\end{figure}

\subsection{The vr-number of $3$-Regular Graphs}
Here we show that the vr-number of a bounded degree graph can be arbitrarily larger than its us-number. In particular,
we show that there exists a $3$-regular graph $G$ with $n$ vertices such that $\vr(G)=\Omega(n)$. Denote by $\pw(G)$ the {\em path-width} of $G$, by $\tw(G)$ the {\em tree-width} of $G$, and by $h(G)$ the \emph{vertex expansion} of $G$. See \cite{BGHK95,GM08} for the definitions of these notions. We need the following known facts:

\begin{equation}\label{eq:pw}
\pw(G)\ge\tw(G)~.
\end{equation}

\begin{lemma}[\cite{BGHK95}] \label{pw}

For every graph $G$:
\[
\vr(G)\ge\pw(G)~.
\]
\end{lemma}

\begin{lemma}[\cite{GM08}] \label{tw}

For every graph $G$:
\[
\tw(G) \geq h(G)\cdot |V(G)|/4~.
\]
\end{lemma}

\begin{mremark}\label{rem:ex}
There exists a constant $h>0$ and an infinite family ${\cal F}$ of $3$-regular graphs, such that for all $G\in {\cal F}$, $h(G)\ge h$ (see, e.g., \cite{HLW06}).
\end{mremark}

Combining \eqref{eq:pw}, \lemmaref{pw} and \lemmaref{tw}, we obtain that the family ${\cal F}$ of \remarkref{rem:ex} satisfies for all $G\in{\cal F}$:
\[
\vr(G) \geq pw(G) \geq tw(G) \geq h(G)\cdot |V(G)|/4=\Omega(|V(G)|)~.
\]

%

\subsubsection*{Acknowledgments.}
We wish to thank Panagiotis Cheilaris for many helpful discussions.


\bibliographystyle{alpha}
\bibliography{paper}

\newcommand{\etalchar}[1]{$^{#1}$}
\begin{thebibliography}{DKKM94}

\bibitem[ACK{\etalchar{+}}04]{ACKKR04}
M.~O. Albertson, G.~G. Chappell, H.~A. Kierstead, A.~K\"{u}ndgen, and
  R.~Ramamurthi.
\newblock Coloring with no $2$-colored {$P_4$'s}.
\newblock {\em The Electronic Journal of Combinatorics}, 11, 2004.

\bibitem[AG06]{AG06}
I.~Abraham and C.~Gavoille.
\newblock Object location using path separators.
\newblock In {\em Proceedings of the twenty-fifth annual ACM symposium on
  Principles of distributed computing}, PODC '06, pages 188--197, New York, NY,
  USA, 2006. ACM.

\bibitem[AH77]{AH77}
K.~Appel and W.~Haken.
\newblock {Every planar map is four colorable.}
\newblock {\em Illinois Journal of Mathematics}, 21(3):429--567, Sep. 1977.

\bibitem[BDJ{\etalchar{+}}98]{BDJ98}
H.~L. Bodlaender, J.~S. Deogun, K.~Jansen, T.~Kloks, D.~Kratsch, H.~M\"{u}ller,
  and Z.~Tuza.
\newblock Rankings of graphs.
\newblock {\em SIAM J. Discret. Math.}, 11(1):168--181, February 1998.

\bibitem[BGHK95]{BGHK95}
H.~L. Bodlaender, J.~R. Gilbert, H.~Hafsteinsson, and T.~Kloks.
\newblock Approximating treewidth, pathwidth, frontsize, and shortest
  elimination tree.
\newblock {\em J. Algorithms}, 18(2):238--255, March 1995.

\bibitem[Bor79]{B79}
O.V. Borodin.
\newblock On acyclic colorings of planar graphs.
\newblock {\em Discrete Mathematics}, 25(3):211--236, 1979.

\bibitem[Deb02]{D02}
M.~Debowsky.
\newblock Results on planar hypergraphs and on cycle decompositions.
\newblock Master's thesis, The University of Vermont, 2002.

\bibitem[DKKM94]{DKKM94}
J.S. Deogun, T.~Kloks, D.~Kratsch, and H.~M\"{u}ller.
\newblock On vertex ranking for permutation and other graphs.
\newblock In Patrice Enjalbert, Ernst~W. Mayr, and KlausW. Wagner, editors,
  {\em STACS 94}, volume 775 of {\em Lecture Notes in Computer Science}, pages
  747--758. Springer Berlin Heidelberg, 1994.

\bibitem[DN06]{DN06}
D.~Dereniowski and A.~Nadolski.
\newblock Vertex rankings of chordal graphs and weighted trees.
\newblock {\em Inf. Process. Lett.}, 98(3):96--100, May 2006.

\bibitem[ES11]{ES11}
G.~Even and S.~Smorodinsky.
\newblock Hitting sets online and vertex ranking.
\newblock In {\em Proceedings of the 19th European conference on Algorithms},
  ESA'11, pages 347--357, Berlin, Heidelberg, 2011. Springer-Verlag.

\bibitem[GM09]{GM08}
M.~Grohe and D.~Marx.
\newblock On tree width, bramble size, and expansion.
\newblock {\em J. Comb. Theory Ser. B}, 99(1):218--228, January 2009.

\bibitem[Gru73]{G73}
B.~Grunbaum.
\newblock Acyclic colorings of planar graphs.
\newblock {\em Israel Journal of Mathematics}, 14(4):390--408, 1973.

\bibitem[HLW06]{HLW06}
S.~Hoory, N.~Linial, and A.~Wigderson.
\newblock Expander graphs and their applications.
\newblock {\em Bull. Amer. Math. Soc.}, 43(4):439--561, 2006.

\bibitem[IRV88]{IRV88}
A.~V. Iyer, H.~D. Ratliff, and G.~Vijayan.
\newblock Optimal node ranking of trees.
\newblock {\em Inf. Process. Lett.}, 28(5):225--229, August 1988.

\bibitem[KMS95]{KMS95}
M.~Katchalski, W.~McCuaig, and S.~Seager.
\newblock Ordered colourings.
\newblock {\em Discrete Mathematics}, 142:141 -- 154, 1995.

\bibitem[Lei80]{L80}
C.~E. Leiserson.
\newblock Area-efficient graph layouts.
\newblock In {\em Proceedings of the 21st Annual Symposium on Foundations of
  Computer Science}, SFCS '80, pages 270--281, Washington, DC, USA, 1980. IEEE
  Computer Society.

\bibitem[Liu90]{L90}
J.~W.~H. Liu.
\newblock The role of elimination trees in sparse factorization.
\newblock {\em SIAM J. Matrix Anal. Appl.}, 11(1):134--172, January 1990.

\bibitem[LT79]{LT79}
R.~J. Lipton and R.~E. Tarjan.
\newblock {A Separator Theorem for Planar Graphs}.
\newblock {\em SIAM Journal on Applied Mathematics}, 36(2):177--189, 1979.

\bibitem[NOdM03]{NO03}
J.~Ne\v{s}et\v{r}il and P.~Ossona~de Mendez.
\newblock Colorings and homomorphisms of minor closed classes.
\newblock In Boris Aronov, Saugata Basu, János Pach, and Micha Sharir,
  editors, {\em Discrete and Computational Geometry}, volume~25 of {\em
  Algorithms and Combinatorics}, pages 651--664. Springer Berlin Heidelberg,
  2003.

\bibitem[Sch88]{S89}
A.~A. Sch\"{a}ffer.
\newblock Optimal vertex ranking of trees in linear time.
\newblock {\em Inf. Process. Lett.}, 33:91--99, 1988.

\bibitem[Smo12]{Smor}
S.~Smorodinsky.
\newblock Conflict-free coloring and its applications.
\newblock In I.~Barany, K.J. Boroczky, G.~Fejes~Toth, and J.~Pach, editors,
  {\em Geometry-Intuitive, Discrete, and Convex}, Bolyai Society Mathematical
  Studies. Springer, 2012.

\bibitem[Tho01]{T01}
A.~Thomason.
\newblock The extremal function for complete minors.
\newblock {\em Journal of Combinatorial Theory, Series B}, 81(2):318 -- 338,
  2001.

\bibitem[Tho04]{T04}
M.~Thorup.
\newblock Compact oracles for reachability and approximate distances in planar
  digraphs.
\newblock {\em J. ACM}, 51(6):993--1024, 2004.

\end{thebibliography}
\end{document}